\documentclass{article}
\usepackage{amssymb,latexsym,amsmath,amsthm,mathrsfs}

\newcommand{\Ei}{\mathrm{Ei}}
\def\Re{\ensuremath{\mathrm{Re}}\,}

\newtheorem{theorem}{Theorem}

\newtheorem{lemma}[theorem]{Lemma}
\begin{document}
\title{Lambert series in analytic number theory}
\author{Jordan Bell \\
jordan.bell@gmail.com}
\date{\today}

\maketitle
\begin{abstract}
Tour of 19th and early 20th century analytic number theory.
\end{abstract}

\section{Introduction}
Let $d(n)$ denote the number of positive divisors of $n$. 
For $|z|<1$,
\[
\sum_{n=1}^\infty d(n) z^n = \sum_{n=1}^\infty \frac{z^n}{1-z^n}.
\]

\section{Euler}
The first use of the term ``Lambert series'' was by Euler to describe the roots of an equation.
 
Euler writes in E25 \cite{E25} about the particular value of a Lambert series.

\section{Lambert}
Bullynck \cite[pp.~157--158]{bullynck}: ``As he recorded in his scientific diary,
the {\em Monatsbuch}, Lambert started
thinking about the divisors of integers in June 1756. An essay by G.W.
Krafft (1701–1754) in the St. Petersburg {\em Novi Commentarii} seems to have
triggered Lambert's interest [Bopp 1916, p. 17, 40].''

Bullynck \cite[p.~163]{bullynck}:

\begin{quote}
Lambert did more than deliver the factor table. He also addressed the
	absence of any coherent theory of prime numbers and divisors. Filling
	such a lacuna could be important for the discovery of new and more primality criteria and
	factoring tests. For Lambert the absence of such a theory was also an occasion to apply the principles
	laid
	out in his philosophical
	work. A fragmentary theory, or one with gaps, needed philosophical and
	mathematical efforts to mature.
	
	\begin{quote}
		To this aim [prime recognition] and others I have looked into the theory
		of prime numbers, but only found certain isolated pieces, which did not seem
		possible to make easily into a connected and well formed system. Euclid has
		few, Fermat some mostly unproven theorems, Euler individual fragments, that
		anyway are farther away from the first beginnings, and leave gaps between them
		and the beginnings. [Lambert 1770, p. 20]
		\end{quote}
\end{quote}

Bullynck \cite[pp.~164--165]{bullynck}:

\begin{quote}
In 1770, Lambert presented two sketches of what would be needed for
something like a theory of numbers. The first dealt mainly with factoring methods [Lambert 1765-1772, II, pp. 1–41], while the second gave a
more axiomatic treatment [Lambert 1770, pp. 20–48]. In the first essay,
Lambert explained how, for composite number with small factors, Eratosthenes'
sieve could be used and optimised. For larger factors, Lambert
explained that approximation from above, starting by division by numbers
that are close to the square root of the tested number $p$, was more advantageous.
For both methods, Lambert advised the use of tables. The
second essay had more theoretical bearings. Lambert rephrased Euclid's
theorems for use in factoring, included the greatest common divisor algorithm, and
put the idea of relatively prime numbers to good use. He also
noted that binary notation, because of the frequent symmetries, could be
helpful. Finally,Lambert also recognized Fermat’s little theorem as a good,
though not infallible criterion for primality, ``but the negative example is
very rar'' [Lambert 1770, p. 43]. 
\end{quote}

Monatsbuch, September 1764, ``Singula haec in Capp. ult. Ontol. occurunt'', and Anm. 5, Anm. 25, 1764, Anm. 12 1765, Anm. 19, 1765 \cite{monatsbuch}.

Lambert  \cite[pp.~506--511, \S 875]{anlage}

Youschkevitch \cite{youschkevitch}

Lorey \cite[p.~23]{lorey}

L\"owenhaupt \cite[p.~32]{lowenhaupt}

\section{Krafft}
Krafft \cite[pp.~244--245]{krafft}

\section{Servois}
Servois \cite{servoisessai} and \cite[p.~166]{servois}

\section{Lacroix}
Lacroix \cite[pp.~465--466, \S 1195]{lacroixIII}

\section{Kl\"ugel}
Kl\"ugel \cite[pp.~52--53, s.v. ``Theiler einer Zahl'', \S 12]{klugel}:

\begin{quote}
Ist $N=\alpha^m \beta^n \gamma^p\cdots$, wo $\alpha,\beta,\gamma$, Primzahlen sind; so erhellet auch leicht, da\ss alle Theiler
von $N$, die Einheit und die Zahl selbst mit engeschlossen, durch die Glieder des Products
\[
(1+\alpha+\alpha^2+\cdots+\alpha^m)(1+\beta+\beta^2+\cdots+\beta^n)(1+\gamma+\gamma^2+\cdots+\gamma^p)\cdots
\]
argestelle werden. Die Anzahl der Glieder dieses Products, d. i. die Anzahl aller Theiler von $N$, ist offenbar 
$=(m+1)(n+1)(p+1)\cdots$. F\"ur das obige Beispiel $=4\cdot 3 \cdot 2=24$, wo die Einheit mit engeschlossen ist.

In der aus der Entwickelung von
\[
\frac{x}{1-x}+\frac{x^2}{1-x^2}+\frac{x^3}{1-x^3}+\cdots+\frac{x^n}{1-x^n}+\cdots
\]
entspringenden Reihe:
\[
x+2x^2+2x^3+3x^4+2x^5+4x^6+2x^7+\cdots
\]
welche Lambert in seiner Architektonik S. 507. mittheilt, enthalt jeder Coefficient so viele Einheiten, als der Exponent der
entsprechendenden Potenz von $x$ Theiler hat.
\end{quote}

\section{Stern}
Stern \cite{stern}

\section{Clausen}
Clausen \cite{clausen} states that
\[
\sum_{n=1}^\infty \frac{x^n}{1-x^n} = \sum_{n=1}^\infty x^{n^2} \left(\frac{1+x^n}{1-x^n}\right),
\]
and that the right-hand series converges quickly for small $x$. Clausen does prove this expansion, and a proof
is later given by Scherk \cite{scherk}. Scherk's argument uses the fact
\[
1+2t+2t^2+2t^3+2t^4+\cdots = 
(1+t+t^2+t^3+t^4+\cdots) + t(1+t+t^2+t^3+t^4+\cdots)=
\frac{1+t}{1-t}.
\]
We write
\[
\sum_{n=1}^\infty \frac{x^n}{1-x^n} = \sum_{n=1}^\infty \sum_{m=1}^\infty x^{nm}.
\]
The series is
\[
\begin{array}{rrrrrrr}
x&+x^2&+x^3&+x^4&+x^5&+x^6&+\mathrm{etc.}\\
+x^2&+x^4&+x^6&+x^8&+x^{10}&+x^{12}&+\mathrm{etc.}\\
+x^3&+x^6&+x^9&x^{12}&+x^{15}&+x^{18}&+\mathrm{etc.}\\
+x^4&+x^8&+x^{12}&+x^{16}&+x^{20}&+x^{24}&+\mathrm{etc.}\\
+x^5&+x^{10}&+x^{15}&+x^{20}&+x^{25}&+x^{30}&+\mathrm{etc.}\\
+x^6&+x^{12}&+x^{18}&+x^{24}&+x^{30}&+x^{36}&+\mathrm{etc.}\\
+\mathrm{etc.}&&&&&&
\end{array}
\]
We sum the terms in the first row and column: the sum of these is
\[
x+2x^2+2x^3+2x^4+\mathrm{etc.} = x\left( \frac{1+x}{1-x}\right).
\]
Then, from what remains we sum the terms in the second row and column: the sum of these is
\[
x^4+2x^6+2x^8+2x^{10}+\mathrm{etc.} = x^4\left( \frac{1+x^2}{1-x^2}\right).
\]
 Then,
from what remains, we sum the terms in the third row and column: the sum of these is
\[
x^9+2x^{12}+2x^{15}+2x^{18}+\mathrm{etc.} = x^9\left(\frac{1+x^3}{1-x^3}\right),
\]
etc.

\section{Eisenstein}
Eisenstein \cite{eisenstein1844} states that
for $|z|<1$,
\[
\sum_{n=1}^\infty \frac{z^n}{1-z^n} =\frac{1}{(1-x)(1-x^2)(1-x^3)\cdots} \sum_{n=1}^\infty 
(-1)^{n+1} \frac{nz^{n(n+1)/2}}{(1-x)\cdots(1-x^n)}.
\]

For $t=\frac{1}{z}$,
Eisenstein states that
\[
\frac{z}{1-z}+\frac{z^2}{1-z^2}+\frac{z^3}{1-z^3}+\frac{z^4}{1-z^4}+\textrm{etc.}
\]
is equal to
\[
\cfrac{1}{t-1-\cfrac{(t-1)^2}{t^2-1-\cfrac{t(t-1)^2}{t^3-1-\cfrac{t(t^2-1)^2}{t^4-1-\cfrac{t^2(t^2-1)^2}{t^5-1-
\cfrac{t^2(t^3-1)^2}{t^6-1-\cfrac{t^3(t^3-1)^2}{t^7-1-\textrm{etc.}}}}}}}}
\]

Expressing Lambert series using continued fractions is relevant to the irrationality of the value of the series. See
Borwein \cite{borwein}. See also Zudilin \cite{zudilin}.

\section{M\"obius}
M\"obius \cite{mobius}

\section{Jacobi}
Jacobi's {\em Fundamenta nova} \cite[\S 40, 66 and p.~185]{fundamenta}

Chandrasekharan \cite[Chapter X]{chandrasekharan}: using Lambert series to prove the four squares theorem.

\section{Dirichlet}
Dirichlet \cite{dirichlet1838}

Fischer \cite{fischer}

\section{Cauchy}
Cauchy \cite{cauchy1843a} and \cite{cauchy1843b} two memoirs in the same volume.

\section{Burhenne}
Burhenne \cite{burhenne} says the following about Lambert series.
For
\[
F(x) = \sum_{n=1}^\infty d(n) x^n,
\]
we have
\[
d(n) = \frac{F^{(n)}(0)}{n!}.
\]
Define
\[
F_k(x) = \frac{x^k}{1-x^k},
\]
so that
\[
F(x) = \sum_{k=1}^\infty F_k(x).
\]
It is apparent that if $k>n$, then
\[
F_k^{(n)}(0) = 0,
\]
hence
\[
F^{(n)}(0)=
 \sum_{k=1}^n F_k^{(n)}(0).
\]

The above suggests finding explicit expressions for $F_k^{(n)}(0)$. 
Burhenne cites Sohncke \cite[pp.~32--33]{sohncke}: for even $k$,
\begin{align*}
\frac{d^n\left(\frac{x^p}{x^k-a^k}\right)}{dx^n}&=(-1)^n \frac{n!}{ka^{k-p-1}}
\left(\frac{1}{(x-a)^{n+1}} - (-1)^p \frac{1}{(x+a)^{n+1}}\right)\\
&+(-1)^n \frac{n!}{\frac{1}{2}ka^{k-p-1}} \sum_{h=1}^{\frac{1}{2}k-1}
\frac{\cos\left(\frac{2h(p+1)\pi}{k}+(n+1)\phi_h\right)}{\sqrt{\left(x^2-2xa\cos \frac{2h\pi}{n}+a^2\right)^{n+1}}}
\end{align*}
and for odd $k$,
\begin{align*}
\frac{d^n\left(\frac{x^p}{x^k-a^k}\right)}{dx^n}&=(-1)^n \frac{n!}{ka^{k-p-1}} \frac{1}{(x-a)^{n+1}}\\
&+(-1)^n \frac{n!}{\frac{1}{2}ka^{k-p-1}} \sum_{h=1}^{\frac{k-1}{2}} \frac{\cos\left( \frac{2h(p+1)\pi}{k}+(n+1)\phi_h \right)}{\sqrt{\left( x^2-2xa\cos \frac{2h\pi}{n}+a^2 \right)^{n+1}}},
\end{align*}
where
\[
\cos \phi_h = \frac{x-a\cos\frac{2h\pi}{k}}{\sqrt{x^2-2xa\cos \frac{2h\pi}{k}+a^2}},
\quad \sin\phi_h = \frac{a\sin \frac{2h\pi}{k}}{\sqrt{x^2-2xa\cos \frac{2h\pi}{k}+a^2}}.
\]
For $a=1$ and $x=0$,
\[
\cos \phi_h = -\cos \frac{2h\pi}{k}, \qquad
\sin \phi_h = \sin \frac{2h \pi}{k},
\]
from which
\[
\phi_h = \pi-\frac{2h\pi}{k},
\]
and thus
\begin{align*}
\cos\left( \frac{2h(k+1)\pi}{k}+(n+1)\phi_h \right)&=
\cos\left( \frac{2h(k+1)\pi}{k}+(n+1)\left( \pi-\frac{2h\pi}{k}\right) \right)\\
&=\cos\left( 2h\pi+\frac{2h\pi}{k}+\pi-\frac{2h\pi}{k}+n\left( \pi-\frac{2h\pi}{k}\right)\right)\\
&=\cos \left( (n+1)\pi- \frac{2nh\pi}{k}\right)\\
&=(-1)^{n+1} \cos \frac{2nh\pi}{k}.
\end{align*}
For even $k$, taking $p=k$ we have
\[
\frac{d^n\left(\frac{x^k}{1-x^k}\right)}{dx^n}=(-1)^{n+1} \frac{n!}{k}\left( \frac{1}{(-1)^{n+1}}-1\right)
+(-1)^{n+1} \frac{n!}{\frac{1}{2}k} \sum_{h=1}^{\frac{1}{2}k-1} (-1)^{n+1} \cos \frac{2nh\pi}{k},
\]
i.e.,
\[
F_k^{(n)}(0) = \frac{n!}{k}(1-(-1)^{n+1})
+\frac{2\cdot n!}{k} \sum_{h=1}^{\frac{1}{2}k-1} \cos \frac{2nh\pi}{k}.
\]
For odd $k$, taking $p=k$ we have
\[
\frac{d^n\left(\frac{x^k}{1-x^k}\right)}{dx^n}=(-1)^{n+1} \frac{n!}{k} \frac{1}{(-1)^{n+1}}
+(-1)^{n+1} \frac{n!}{\frac{1}{2}k} \sum_{h=1}^{\frac{k-1}{2}} (-1)^{n+1} \cos \frac{2nh\pi}{k},
\]
i.e.,
\[
F_k^{(n)}(0) = \frac{n!}{k}+\frac{2\cdot n!}{k} \sum_{h=1}^{\frac{k-1}{2}} \cos \frac{2nh\pi}{k}.
\]
Using the identity, for $h \not \in 2\pi \mathbb{Z}$,
\[
\sum_{h=1}^M \cos h\theta = - \frac{1}{2}+\frac{\sin\left(M+\frac{1}{2}\right)\theta}{2\sin\frac{\theta}{2}}
=-\frac{1}{2}+\frac{1}{2}\left(\sin M\theta \cot \frac{\theta}{2}+\cos M \theta\right),
\] 
we  get for even $k$,
\begin{align*}
F_k^{(n)}(0) &= \begin{cases}
\frac{n!}{k} \cot \frac{n\pi}{k} \sin n\pi&k \not | n\\
\frac{n!}{k}(1-(-1)^{n+1})
+\frac{2\cdot n!}{k} \left(\frac{1}{2}k-1\right)
&k  | n
\end{cases}\\
&=\begin{cases}
0&k \not | n\\
n!-\frac{n!}{k}(1+(-1)^{n+1})& k| n.
\end{cases}
\end{align*}
For odd $k$,
\begin{align*}
F_k^{(n)}(0) 
&=\begin{cases}
\frac{n!}{k} \csc \frac{n\pi}{k} \sin n\pi&k \not | n\\
\frac{n!}{k}+\frac{2\cdot n!}{k} \frac{k-1}{2}&k | n.
\end{cases}\\
&=\begin{cases}
0&k \not | n\\
n!&k|n.
\end{cases}
\end{align*}

\section{Zehfuss}
Zehfuss \cite{zehfuss}

\section{Bernoulli numbers}
The \textbf{Bernoulli polynomials} are defined by
\[
\frac{te^{tx}}{e^t-1} = \sum_{m=0}^\infty B_m(x) \frac{t^m}{m!}.
\]
The \textbf{Bernoulli numbers} are defined by $B_m=B_m(0)$. 

We denote by $[x]$ the greatest integer $\leq x$, and we define $\{x\}=x-[x]$, namely, the fractional part of $x$.
We define $P_m(x)=B_m(\{x\})$, the \textbf{periodic Bernoulli functions}.

\section{Euler-Maclaurin summation formula}
Euler E47 and E212, \S 142, for the summation formula. Euler's studies the gamma function in E368. In particular,
in \S 12 he gives Stirling's formula, and
in \S 14 he obtains $\Gamma'(1)=-\gamma$. 
Euler in \S 142 of E212 states that
\[
\gamma = \frac{1}{2}+\sum_{n=1}^\infty \frac{(-1)^{n+1}B_{2n}}{2n}.
\]

Bromwich \cite[Chapter~XII]{bromwich}

The Euler-Maclaurin summation formula \cite[p.~280, Ch.~VI, Eq.~35]{bourbaki} tells us that for $f \in C^\infty([0,1])$,
\[
f(0)=\int_0^1 f(t) dt
+B_1(f(1)-f(0))+
\sum_{m=1}^k \frac{1}{(2m)!} B_{2m} ( f^{(2m-1)}(1)-f^{(2m-1)}(0))+R_{2k},
\]
where 
\[
R_{2k} = -\int_0^1 \frac{P_{2k}(1-\eta)}{(2k)!} f^{(2k)}(\eta) d\eta.
\]

Poisson and Jacobi  on the Euler-Maclaurin summation formula.

\section{Schl\"omilch}
Schl\"omilch \cite{schlomilch1861}  and  \cite[p.~238]{compendium}, \cite{schlomilch1863}

For $m \geq 1$,
\begin{equation}
\int_0^\infty \frac{t^{2m-1}}{e^{2\pi t}-1} dt=(-1)^{m+1} \frac{B_{2m}}{4m}.
\label{bernoulliInt}
\end{equation}
For $\alpha>0$,
\begin{equation}
\int_0^\infty \frac{\sin \alpha t}{e^{2\pi t}-1} dt = \frac{1}{4}+\frac{1}{2}\left(\frac{1}{e^\alpha-1}-\frac{1}{\alpha}\right)
\label{sinalpha}
\end{equation}
and
\begin{equation}
\int_0^\infty \frac{1-\cos \alpha t}{e^{2\pi t}-1} \frac{dt}{t} = \frac{1}{4}\alpha+\frac{1}{2}\left(\log(1-e^{-\alpha})-\log \alpha\right).
\label{cosalpha}
\end{equation}

For $\xi>0$ and $n \geq 1$, using \eqref{sinalpha} with $\alpha=\xi,2\xi,3\xi,\ldots,2n\xi$ 
and also using
\[
\sum_{k=1}^N \sin k\theta = \frac{1}{2}\cot \frac{\theta}{2}-\frac{\cos(N+\frac{1}{2})\theta}{2\sin \frac{\theta}{2}},
\]
we get
\begin{align*}
\sum_{m=1}^{2n} \left(\frac{1}{e^{m\xi}-1} -\frac{1}{m\xi} \right)&=\sum_{m=1}^{2n} \left(-\frac{1}{2}+2\int_0^\infty \frac{\sin m\xi t}{e^{2\pi t}-1} dt \right)\\
&=-n + \int_0^\infty \frac{1}{e^{2\pi t}-1} \sum_{m=1}^{2n} 2\sin m\xi t dt\\
&=-n+\int_0^\infty \frac{1}{e^{2\pi t}-1} \left(\cot \frac{\xi t}{2}-\frac{\cos(2n+\frac{1}{2})\xi t}{\sin \frac{\xi t}{2}} \right) dt.
\end{align*}
Using $\cos(a+b)=\cos a \cos b - \sin a \sin b$, this becomes
\begin{align}
\sum_{m=1}^{2n} \left(\frac{1}{e^{m\xi}-1} -\frac{1}{m\xi} \right)&=
-n+\int_0^\infty \frac{1}{e^{2\pi t}-1} (1-\cos 2n\xi t)\cot \frac{\xi t}{2}dt \label{xisum}\\
& +\int_0^\infty \frac{1}{e^{2\pi t}-1}  \sin 2n\xi t dt.
\nonumber
\end{align}

For $\alpha = 2n\xi$, \eqref{cosalpha} tells us
\[
\int_0^\infty \frac{1-\cos 2n\xi t}{e^{2\pi t}-1} \frac{dt}{t} = \frac{1}{4}\cdot 2n\xi + \frac{1}{2}\left( \log(1-e^{-2n\xi})-\log 2n\xi\right).
\]
Rearranging, 
\begin{equation}
\frac{\log 2n}{\xi} = n + \frac{\log(1-e^{-2n \xi})-\log \xi}{\xi} -\frac{2}{\xi}
\int_0^\infty \frac{1-\cos 2n\xi t}{e^{2\pi t}-1} \frac{dt}{t}
\label{log2n}
\end{equation}
Adding \eqref{xisum} and \eqref{log2n} gives
\[
\begin{split}
&\sum_{m=1}^{2n} \frac{1}{e^{m\xi}-1} - \frac{1}{\xi} \left( - \log 2n+\sum_{m=1}^{2n} \frac{1}{m}  \right)\\
=& \frac{\log(1-e^{-2n \xi})-\log \xi}{\xi}
-\int_0^\infty \left(\frac{2}{\xi t}-\cot \frac{\xi t}{2} \right) \frac{1-\cos 2n\xi t}{e^{2\pi t}-1} dt\\
&+\int_0^\infty \frac{1}{e^{2\pi t}-1}  \sin 2n\xi t dt.
\end{split}
\]
Writing
\[
C_n = -\log n + \sum_{m=1}^n \frac{1}{m}
\]
and using \eqref{sinalpha} this becomes
\[
\begin{split}
&\sum_{m=1}^{2n} \frac{1}{e^{m\xi}-1} - \frac{1}{\xi} C_{2n}\\
=&\frac{\log(1-e^{-2n \xi})-\log \xi}{\xi}
-2\int_0^\infty \left(\frac{1}{\xi t}-\frac{1}{2}\cot \frac{\xi t}{2} \right) \frac{1-\cos 2n\xi t}{e^{2\pi t}-1} dt\\
&+\frac{1}{4}+\frac{1}{2}\left( \frac{1}{e^{2n \xi}-1}-\frac{1}{2n\xi}\right).
\end{split}
\]
We write
\[
I_{2n}(\xi) = 2\int_0^\infty \left(\frac{1}{\xi t}-\frac{1}{2}\cot \frac{\xi t}{2} \right) \frac{1-\cos 2n\xi t}{e^{2\pi t}-1} dt,
\]
and we shall obtain an asymptotic formula for $I_{2n}(\xi)$.

We apply the Euler-Maclaurin summation formula.
Let $h>0$, and for $f(t)=\cos ht$ we have $f'(t)=-h\sin ht$, and for $m \geq 1$ we have
$f^{(2m)}(t)=(-1)^m h^{2m} \cos ht$ and $f^{(2m-1)}(t)=(-1)^m h^{2m-1} \sin ht$. 
Thus the Euler-Maclaurin formula yields
\[
1 = \int_0^1 \cos ht dt - \frac{1}{2}(\cos h - 1)
+\sum_{m=1}^k \frac{1}{(2m)!} B_{2m} (-1)^m h^{2m-1} \sin h + R_{2k}.
\]
Using the identity $\cot \frac{\theta}{2} = \frac{1+\cos \theta}{\sin \theta}$ and dividing by $\sin h$, this becomes
\begin{equation}
\frac{1}{2} \cot \frac{h}{2} =\frac{1}{h} 
+\sum_{m=1}^k \frac{1}{(2m)!} B_{2m} (-1)^m h^{2m-1} + \frac{1}{\sin h} R_{2k}.
\label{EMcos}
\end{equation}
Because $P_m(1-\eta)=P_m(\eta)$ for even $m$,
\begin{align*}
R_{2k}& = - \int_0^1 \frac{P_{2k}(\eta)}{(2k)!}(-1)^k h^{2k} \cos h\eta d\eta\\
&=- B_{2k} \int_0^1 \frac{1}{(2k)!}(-1)^k h^{2k} \cos h\eta d\eta
- \int_0^1 \frac{(P_{2k}(\eta)-B_{2k})}{(2k)!}(-1)^k h^{2k} \cos h\eta d\eta
\\
&=(-1)^{k+1} \frac{B_{2k} h^{2k}}{(2k)!}  \frac{\sin h}{h}
+(-1)^{k+1} \frac{h^{2k}}{(2k)!} \int_0^1 (P_{2k}(\eta)-B_{2k})  \cos h\eta d\eta.
\end{align*}
Since $P_{2k}(\eta)-B_{2k}$ does not change sign on $(0,1)$, by the mean-value theorem for integration there
is some $\theta=\theta(h,k)$, $0<\theta<1$, such that (using 
$\int_0^1 P_{2k}(\eta) d\eta=0$)
\[
 \int_0^1 (P_{2k}(\eta)-B_{2k})  \cos h\eta d\eta=
 \cos h\theta  \int_0^1 (P_{2k}(\eta)-B_{2k}) d\eta 
 =-B_{2k} \cos h\theta.
\]
Therefore \eqref{EMcos} becomes
\begin{align*}
\frac{1}{2} \cot \frac{h}{2} - \frac{1}{h}&=
\sum_{m=1}^k \frac{1}{(2m)!} B_{2m} (-1)^m h^{2m-1}\\
&+(-1)^{k+1} \frac{B_{2k} h^{2k-1}}{(2k)!}+(-1)^{k+2} \frac{h^{2k}}{(2k)! \sin h}B_{2k} \cos h\theta,
\end{align*}
i.e., 
\[
\frac{1}{2} \cot \frac{h}{2} - \frac{1}{h}=
\sum_{m=1}^{k-1} \frac{1}{(2m)!} B_{2m} (-1)^m h^{2m-1}+(-1)^k \frac{h^{2k}}{(2k)! \sin h}B_{2k} \cos h\theta.
\]

Write
\[
E_k(h) = (-1)^{k+1} \frac{h^{2k}}{(2k)! \sin h}B_{2k} \cos h\theta.
\]
We apply the above to $I_{2n}(\xi)$, and get, for any $k \geq 1$,
\begin{align*}
I_{2n}(\xi)&=2\int_0^\infty \left(E_k(\xi t)-
\sum_{m=1}^{k-1} \frac{1}{(2m)!} B_{2m} (-1)^m (\xi t)^{2m-1} \right) \frac{1-\cos 2n\xi t}{e^{2\pi t}-1} dt\\
&= -2\sum_{m=1}^{k-1} \frac{1}{(2m)!} B_{2m} (-1)^m \xi^{2m-1} \int_0^\infty  t^{2m-1} 
 \frac{1-\cos 2n\xi t}{e^{2\pi t}-1} dt\\
 &+2 \int_0^\infty E_k(\xi t)  \frac{1-\cos 2n\xi t}{e^{2\pi t}-1} dt.
\end{align*} 
Using \eqref{bernoulliInt},
\begin{align*}
 \int_0^\infty  t^{2m-1} 
 \frac{1-\cos 2n\xi t}{e^{2\pi t}-1} dt&=
 \int_0^\infty \frac{t^{2m-1}}{e^{2\pi t}-1} dt
 -\int_0^\infty \frac{t^{2m-1} \cos 2n\xi t}{e^{2\pi t}-1} dt\\
 &=(-1)^{m+1} \frac{B_{2m}}{4m} -\int_0^\infty \frac{t^{2m-1} \cos 2n\xi t}{e^{2\pi t}-1} dt.
\end{align*}
Let
\[
f(x)=\frac{1}{e^x-1}-\frac{1}{x}.
\]
By \eqref{sinalpha},
\[
f(x)+\frac{1}{2} = 2 \int_0^\infty \frac{\sin x t}{e^{2\pi t}-1} dt.
\]
For $m \geq 1$, 
\[
f^{(2m-1)}(x) = 2\int_0^\infty \frac{(-1)^{m-1} t^{2m-1} \cos x t}{e^{2\pi t}-1} dt,
\]
which for $x=2n\xi$ becomes
\[
\frac{(-1)^{m-1}}{2} f^{(2m-1)}(2n\xi) = \int_0^\infty \frac{t^{2m-1} \cos 2n\xi t}{e^{2\pi t}-1} dt.
\]
Therefore
\[
2 \int_0^\infty  t^{2m-1} 
 \frac{1-\cos 2n\xi t}{e^{2\pi t}-1} dt
 =(-1)^{m+1} \frac{B_{2m}}{2m}+(-1)^{m} f^{(2m-1)}(2n\xi).
\]
Thus $I_{2n}(\xi)$ is
\begin{align*}
I_{2n}(\xi)&=-\sum_{m=1}^{k-1} \frac{1}{(2m)!} B_{2m} (-1)^m \xi^{2m-1} 
\left(
(-1)^{m+1} \frac{B_{2m}}{2m}+(-1)^{m} f^{(2m-1)}(2n\xi)
\right)
\\
 &+2 \int_0^\infty E_k(\xi t)  \frac{1-\cos 2n\xi t}{e^{2\pi t}-1} dt\\
 &=\sum_{m=1}^{k-1} \frac{B_{2m}^2}{(2m)! 2m} \xi^{2m-1}
 -\sum_{m=1}^{k-1} \frac{B_{2m}}{(2m)!} \xi^{2m-1} f^{(2m-1)}(2n\xi)\\
&+2 \int_0^\infty E_k(\xi t)  \frac{1-\cos 2n\xi t}{e^{2\pi t}-1} dt.
\end{align*}
But
\begin{align*}
\left| \int_0^\infty E_k(\xi t)  \frac{1-\cos 2n\xi t}{e^{2\pi t}-1} dt\right|&=
\left| \int_0^\infty (-1)^{k+1} \frac{(\xi t)^{2k}}{(2k)! \sin \xi t}B_{2k} \cos \xi t \theta 
 \frac{1-\cos 2n\xi t}{e^{2\pi t}-1} dt \right|\\
 &\leq \frac{|B_{2k}|}{(2k)!}  \int_0^\infty  \frac{(\xi t)^{2k}}{|\sin \xi t|} \frac{1-\cos 2n\xi t}{e^{2\pi t}-1} dt.
 \end{align*}
It is a fact that for all $u \in \mathbb{R}$,
\[
\frac{1-\cos 2n u}{|\sin u|} \leq \frac{\pi}{2} \frac{1-\cos 2n u}{u},
\]
we obtain
\[
\begin{split}
&\left| \int_0^\infty E_k(\xi t)  \frac{1-\cos 2n\xi t}{e^{2\pi t}-1} dt\right|\\
\leq& \frac{\pi}{2} \frac{|B_{2k}|}{(2k)!} \int_0^\infty (\xi t)^{2k-1} \frac{1-\cos 2n\xi t}{e^{2\pi t}-1} dt\\
=& \frac{\pi}{2} \frac{|B_{2k}|}{(2k)!} \xi^{2k-1} \cdot \frac{1}{2}
\left(
(-1)^{k+1} \frac{B_{2k}}{2k}+(-1)^{k} f^{(2k-1)}(2n\xi)
\right).
\end{split}
\]
Hence
\begin{align*}
I_{2n}(\xi)&=\sum_{m=1}^{k-1} \frac{B_{2m}^2}{(2m)! 2m} \xi^{2m-1}
 -\sum_{m=1}^{k-1} \frac{B_{2m}}{(2m)!} \xi^{2m-1} f^{(2m-1)}(2n\xi)\\
 &+O\left( \frac{B_{2k}^2}{(2k)!2k} \xi^{2k-1}\right)
 +O\left( \frac{|B_{2k}|}{(2k)!} \xi^{2k-1} f^{(2k-1)}(2n\xi) \right).
\end{align*}
Therefore we have
\[
\begin{split}
&\sum_{m=1}^{2n} \frac{1}{e^{m\xi}-1} - \frac{1}{\xi} C_{2n}\\
=&\frac{\log(1-e^{-2n \xi})-\log \xi}{\xi}
+\frac{1}{4}+\frac{1}{2}\left( \frac{1}{e^{2n \xi}-1}-\frac{1}{2n\xi}\right)-I_{2n}(\xi)\\
=&\frac{\log(1-e^{-2n \xi})-\log \xi}{\xi}+\frac{1}{4}+\frac{1}{2}\left( \frac{1}{e^{2n \xi}-1}-\frac{1}{2n\xi}\right)\\
&-\sum_{m=1}^{k-1} \frac{B_{2m}^2}{(2m)! 2m} \xi^{2m-1}
+\sum_{m=1}^{k-1} \frac{B_{2m}}{(2m)!} \xi^{2m-1} f^{(2m-1)}(2n\xi)\\
&+O\left( \frac{B_{2k}^2}{(2k)!2k} \xi^{2k-1}\right)
 +O\left( \frac{|B_{2k}|}{(2k)!} \xi^{2k-1} f^{(2k-1)}(2n\xi) \right).
\end{split}
\]
Taking $n \to \infty$,
\[
\sum_{m=1}^\infty \frac{1}{e^{m\xi}-1} - \frac{\gamma}{\xi}=-\frac{\log \xi}{\xi}
+\frac{1}{4}-\sum_{m=1}^{k-1} \frac{B_{2m}^2}{(2m)! 2m} \xi^{2m-1}+O\left( \frac{B_{2k}^2}{(2k)!2k} \xi^{2k-1}\right).
\]

\section{Voronoi summation formula}
The Voronoi summation formula  \cite[p.~182]{cohen} states that if $f:\mathbb{R} \to \mathbb{C}$
is a Schwartz function, then
\begin{align*}
\sum_{n=1}^\infty d(n) f(n)&=\int_0^\infty f(t)(\log t + 2\gamma) dt + \frac{f(0)}{4}\\
&+\sum_{n=1}^\infty d(n) \int_0^\infty f(t)(4K_0(4\pi(nt)^{1/2})-2\pi Y_0(4\pi(nt)^{1/2})) dt,
\end{align*}
where $K_0$ and $Y_0$ are Bessel functions. 

Let $0<x<1$. For $f(t)=e^{-tx}$, we compute 
\[
\begin{split}
&\int_0^\infty f(t)(4K_0(4\pi(nt)^{1/2})-2\pi Y_0(4\pi(nt)^{1/2})) dt\\
=&
-\frac{2}{x}\exp\left(\frac{4\pi^2 n}{x}\right) \Ei\left(-\frac{4\pi^2 n}{x}\right)
-\frac{2}{x}\exp\left(-\frac{4\pi^2 n}{x}\right) \Ei\left(\frac{4\pi^2 n}{x}\right),
\end{split}
\]
where
\[
\Ei(x) = -\int_{-x}^\infty \frac{e^{-t}}{t} dt, \qquad x \neq 0,
\]
the exponential integral.
Then the Voronoi summation formula yields
\[
\begin{split}
&\sum_{n=1}^\infty d(n) e^{-nx}\\
=& \frac{\gamma}{x} - \frac{\log x}{x} + \frac{1}{4}\\
&+ \sum_{n=1}^\infty d(n)
\left(
-\frac{2}{x}\exp\left(\frac{4\pi^2 n}{x}\right) \Ei\left(-\frac{4\pi^2 n}{x}\right)
-\frac{2}{x}\exp\left(-\frac{4\pi^2 n}{x}\right) \Ei\left(\frac{4\pi^2 n}{x}\right)
\right).
\end{split}
\]

Egger and Steiner \cite{egger} give a proof  of the Voronoi summation formula involving Lambert series.

Kluyver \cite{kluyver1919} and \cite{kluyver1922}

Guinand \cite{guinand}

\section{Curtze}
Curtze \cite{curtze}

\section{Laguerre}
Laguerre \cite{laguerre}

\section{V. A. Lebesgue}
V. A. Lebesgue \cite{lebesgue}:

\section{Bouniakowsky}
Bouniakowsky \cite{bouniakowsky}

\section{Chebyshev}
Chebyshev \cite{tchebychef}

\section{Catalan}
Catalan \cite{catalan1842}

Catalan \cite[p.~89]{catalan1873}

Catalan \cite[p.~119, \S CXXIV]{melanges1886} and \cite[pp.~38--39, \S CCXXVI]{melanges1888}

\section{Pincherle}
Pincherle \cite{pincherle}

\section{Glaisher}
Glaisher \cite[p.~163]{glaisher}

\section{G\"unther}
G\"unther \cite[p.~83]{gunther} and \cite[p.~178]{gunther1881}

\section{Stieltjes}
Stieltjes \cite{stieltjes}

cf. Zhang \cite{zhang}

\section{Rogel}
Rogel \cite{rogel} and \cite{rogel1891}

\section{Ces\`aro}
Ces\`aro \cite{cesaro1886}

Ces\`aro \cite{cesaro1888}

Ces\`aro \cite{cesaro1893} and \cite[pp.~181--184]{cesaro1894}

Bromwich \cite[p.~201, Chapter VIII, Example B, 35]{bromwich}

\section{de la Vall\'ee-Poussin}
de la Vall\'ee-Poussin \cite{valleepoussin}

\section{Torelli}
Torelli \cite{torelli}

\section{Fibonacci numbers}
Landau \cite{fibonacci}

\section{Knopp}
Knopp \cite{knopp1913}

\section{Generating functions}
Hardy and Wright \cite[p.~258, Theorem 307]{wright}:

\begin{theorem}
For $f(s)=\sum_{n=1}^\infty a_n n^{-s}$ and $g(s)=\sum_{n=1}^\infty b_n n^{-s}$, 
\[
\sum_{n=1}^\infty a_n \frac{x^n}{1-x^n} = \sum_{n=1}^\infty b_n x^n, \qquad |x|<1,
\]
if and only if there is some $\sigma$ such that
\[
\zeta(s) f(s) = g(s), \qquad \Re(s)>\sigma.
\]
\end{theorem}

For $f(s)=\sum_{n=1}^\infty \mu(n) n^{-s}$ and $g(s)=1$,
using \cite[p.~250, Theorem 287]{wright}
\[
\frac{1}{\zeta(s)} = \sum_{n=1}^\infty \mu(n) n^{-s}, \qquad \Re(s)>1,
\]
we get
\begin{equation}
\sum_{n=1}^\infty  \frac{\mu(n) x^n}{1-x^n} = x.
\label{mobiusgenerating}
\end{equation}

For $f(s) = \sum_{n=1}^\infty \phi(n) n^{-s}$ and 
\[
g(s)=\zeta(s-1)=\sum_{n=1}^\infty n^{-s+1} = \sum_{n=1}^\infty n n^{-s},
\]
using \cite[p.~250, Theorem 288]{wright}
\[
\frac{\zeta(s-1)}{\zeta(s)} = \sum_{n=1}^\infty  \phi(n) n^{-s}, \qquad \Re(s)>2,
\]
we get
\[
\sum_{n=1}^\infty  \frac{\phi(n) x^n}{1-x^n} = \sum_{n=1}^\infty nx^n =
\frac{x}{(1-x)^2}.
\]

For $n=p_1^{a_1} \cdots p_r^{a_r}$, define $\Omega(n) = a_1+\cdots+a_n$ and 
\[
\lambda(n)=(-1)^{\Omega(n)}.
\]
For $f(s)= \sum_{n=1}^\infty \lambda(n) n^{-s}$ and
\[
g(s)=\zeta(2s) = \sum_{n=1}^\infty n^{-2s} = \sum_{n=1}^\infty (n^2)^{-s},
\]
using \cite[p.~255, Theorem 300]{wright}
\[
\frac{\zeta(2s)}{\zeta(s)} = \sum_{n=1}^\infty \lambda(n) n^{-s}, \qquad \Re(s)>1,
\]
we get
\[
\sum_{n=1}^\infty \frac{\lambda(n) x^n}{1-x^n} = \sum_{n=1}^\infty x^{n^2}.
\]

We define the \textbf{von Mangoldt function} $\Lambda:\mathbb{N} \to \mathbb{R}$ by
$\Lambda(n)=\log p$ if $n$ is some positive integer power of a prime $p$, and $\Lambda(n)=0$ otherwise. For example,
$\Lambda(1)=0$, $\Lambda(12)=0$, $\Lambda(125)=\log 5$. It is a fact \cite[p.~254, Theorem 296]{wright} that for any $n$, the von Mangoldt function satisfies
\begin{equation}
\sum_{m | n} \Lambda(m) = \log n.
\label{mangoldtsum}
\end{equation}

For $f(s)=\sum_{n=1}^\infty \Lambda(n) n^{-s}$ and
\[
g(s) = -\zeta'(s) = \sum_{n=1}^\infty \log n n^{-s},
\]
using \cite[p.~253, Theorem 294]{wright}
\[
-\frac{\zeta'(s)}{\zeta(s)} = \sum_{n=1}^\infty \Lambda(n) n^{-s},
\]
we obtain
\[
\sum_{n=1}^\infty \frac{\Lambda(n) x^n}{1-x^n} = \sum_{n=1}^\infty \log n x^n.
\]

\section{Mertens}
For $\Re s>1$, we define
\[
P(s) = \sum_p \frac{1}{p^s}.
\]
We also 
define
\[
H = \sum_{m=2}^\infty \sum_p \frac{1}{mp^m}.
\]

Mertens \cite{mertens} proves the following.

\begin{theorem}
As $\varrho \to 0$,
\[
P(1+\rho) = \log \left( \frac{1}{\rho} \right)-H+o(1).
\]
\end{theorem}
\begin{proof}
As $\varrho \to 0$,
\[
\zeta(1+\varrho) = \frac{1}{\varrho}+\gamma+O(\varrho) = \frac{1}{\varrho}(1+\gamma \varrho + O(\varrho^2)).
\]
Taking the logarithm,
\begin{equation}
\log \zeta(1+\varrho) = \log \left(\frac{1}{\varrho} \right) + \log(1+\gamma \varrho + O(\varrho^2))
= \log \left(\frac{1}{\varrho} \right)  + \gamma \varrho+O(\varrho^2).
\label{logzetaeq}
\end{equation}
On the other hand, for $\varrho>0$,
\[
\zeta(1+\varrho) = \prod_p \frac{1}{1-\frac{1}{p^{1+\varrho}}},
\]
and taking the logarithm,
\begin{align*}
\log \zeta(1+\varrho) &= -  \sum_p \log\left(1-\frac{1}{p^{1+\varrho}}\right)\\
&=\sum_p \sum_{m=1}^\infty \frac{1}{mp^{m(1+\varrho)}}\\
&=P(1+\rho) + \sum_{m=2}^\infty \sum_p  \frac{1}{mp^{m(1+\varrho)}}.
\end{align*}

Then as $\varrho \to 0$,
\[
\log \zeta(1+\varrho)  = P(1+\varrho) + H + o(1).
\]
Combining this with \eqref{logzetaeq} we get that as $\varrho \to 0$,
\[
P(1+\rho) = \log \left( \frac{1}{\rho} \right)-H+o(1).
\]
\end{proof}

Mertens \cite{mertens} also proves that for any $x$ there is some 
\[
|\delta|< \frac{4}{\log(x+1)} + \frac{2}{x\log x}
\]
such that
\[
\sum_{p \leq x} \frac{1}{p} = \log \log x + \gamma - H + \delta.
\]
Thus,
\[
\sum_{p \leq x} \frac{1}{p} = \log \log x + \gamma - H + O\left( \frac{1}{\log x} \right).
\]

Mertens shows that 
\[
H=-\sum_{n=2}^\infty \mu(n) \frac{\log \zeta(n)}{n}.
\]
This can be derived using \eqref{mobiusgenerating}, and we do this now; see \cite{lindqvist}.

\begin{lemma}
For $\Re s>1$,
\[
\frac{1}{s} \log \zeta(s) = \int_2^\infty \frac{\pi(t) dt}{t(t^s-1)}.
\]
\label{logzetalemma}
\end{lemma}
\begin{proof}
For $p$ prime and $\Re s>0$,
\begin{align*}
\int_p^\infty \frac{dt}{t(t^s-1)}&=
\int_p^\infty t^{-s-1}\frac{1}{1-t^{-s}} dt\\
&=\int_p^\infty t^{-s-1} \sum_{n=0}^\infty (t^{-s})^n dt\\
&=\sum_{n=0}^\infty \int_p^\infty t^{-ns-s-1} dt\\
&=\sum_{n=0}^\infty \frac{t^{-ns-s}}{-ns-s} \bigg|_p^\infty\\
&=\frac{1}{s} \sum_{n=1}^\infty \frac{p^{-ns}}{n}\\
&=-\frac{1}{s} \log(1-p^{-s}),
\end{align*}
hence
\[
 \log \left( \frac{1}{1-p^{-s}} \right) = s  \int_p^\infty \frac{dt}{t(t^s-1)}.
\]
On the one hand,
\[
\sum_p  \int_p^\infty \frac{dt}{t(t^s-1)} = \int_2^\infty \frac{\pi(t) dt}{t(t^s-1)}.
\]
On the other hand, for $\Re s>1$ we have
\[
\sum_p  \log \left( \frac{1}{1-p^{-s}} \right) = \log \prod_p  \left( \frac{1}{1-p^{-s}} \right)
=\log \zeta(s). 
\]
Combining these, for $\Re s>1$,
\[
\frac{1}{s} \log \zeta(s) = \int_2^\infty \frac{\pi(t) dt}{t(t^s-1)}.
\]
\end{proof}

\begin{theorem}
\[
H=-\sum_{n=2}^\infty \mu(n) \frac{\log \zeta(n)}{n}.
\]
\end{theorem}
\begin{proof}
For any prime $p$ and for $m \geq 1$,
\[
\int_p^\infty t^{-m-1} dt  = 
\frac{t^{-m}}{-m} \bigg|_p^\infty =
\frac{1}{mp^m}, 
\]
and using this we have
\begin{align*}
H&=\sum_{m=2}^\infty \sum_p \frac{1}{mp^m}\\
&=\sum_{m=2}^\infty \sum_p \int_p^\infty t^{-m-1} dt \\
&=\sum_{m=2}^\infty \int_2^\infty \pi(t) t^{-m-1} dt  \\
&=\int_2^\infty \pi(t) \left( \sum_{m=2}^\infty t^{-m-1}\right) dt\\
&=\int_2^\infty \pi(t) \frac{1}{t^2(t-1)}  dt
\end{align*}

Rearranging \eqref{mobiusgenerating}, 
\[
\frac{x^2}{1-x}=
-\sum_{n=2}^\infty  \frac{\mu(n) x^n}{1-x^n}.
\]
With $x=t^{-1}$,
\[
\frac{1}{t(t-1)} = -\sum_{n=2}^\infty \frac{\mu(n)}{t^n-1},
\]
so
\[
\frac{1}{t^2(t-1)} = -\sum_{n=2}^\infty \frac{\mu(n)}{t(t^n-1)}.
\]
Thus we have
\[
H=-\int_2^\infty \pi(t) \left( \sum_{n=2}^\infty \frac{\mu(n)}{t(t^n-1)}\right) dt
=-\sum_{n=2}^\infty \mu(n) \int_2^\infty \frac{\pi(t) dt}{t(t^n-1)} dt.
\]
Using  Lemma \ref{logzetalemma} for $s=2,3,4,\ldots$,
\[
H=-\sum_{n=2}^\infty \mu(n) \cdot \frac{1}{n} \log \zeta(n),
\]
completing the proof.
\end{proof}

\section{Preliminaries on prime numbers}
We define
\[
\vartheta(x) = \sum_{p \leq x} \log p = \log \prod_{p \leq x} p
\]
and
\[
\psi(x) = \sum_{p^m \leq x} \log p = \sum_{n \leq x} \Lambda(n).
\]
One sees that
\[
\psi(x) = \sum_{p \leq x} [\log_p x] \log p = 
\sum_{p \leq x} \left[ \frac{\log x}{\log p}\right] \log p.
\]
As well,
\begin{equation}
\psi(x)
= 
\sum_{m=1}^\infty \sum_{p \leq x^{1/m}} \log p
=
\sum_{m=1}^\infty \vartheta(x^{1/m});
\label{thetasum}
\end{equation}
there are only finitely many terms on the right-hand side, as $\vartheta(x^{1/m})=0$ if
$x<2^m$.

\begin{theorem}
\[
\psi(x) = \vartheta(x)+O(x^{1/2} (\log x)^2).
\]
\label{psitheta}
\end{theorem}
\begin{proof}
For $x \geq 2$,
$\vartheta(x)<x \log x$, giving
\begin{align*}
 \sum_{2 \leq m 
\leq \frac{\log x}{\log 2}} \vartheta(x^{1/m})
&< \sum_{2 \leq m 
\leq \frac{\log x}{\log 2}} x^{1/m} \frac{1}{m} \log x\\
&\leq x^{1/2} \log x \sum_{2 \leq m \leq \frac{\log x}{\log 2}} \frac{1}{m}\\
&=O(x^{1/2}(\log x)^2).
\end{align*}
Thus, using \eqref{thetasum} we have
\[
\psi(x) = \vartheta(x) +  \sum_{2 \leq m 
\leq \frac{\log x}{\log 2}} \vartheta(x^{1/m})=
\vartheta(x)+O(x^{1/2}(\log x)^2).
\]
\end{proof}

We prove that if $\lim_{x \to \infty} \frac{\vartheta(x)}{x}=1$ then
$\frac{\pi(x)}{x/\log x}=1$.

\begin{theorem}
\[
\liminf_{x \to \infty} \frac{\pi(x)}{x/\log x}
=\liminf_{x \to \infty} \frac{\vartheta(x)}{x}
\]
and
\[
\limsup_{x \to \infty} \frac{\pi(x)}{x/\log x}
=\limsup_{x \to \infty} \frac{\vartheta(x)}{x}.
\]
\end{theorem}
\begin{proof}
From \eqref{thetasum}, $\vartheta(x) \leq \psi(x)$. And,
\[
\psi(x)
=\sum_{p \leq x} \left[ \frac{\log x}{\log p}\right] \log p
\leq \sum_{p \leq x} \frac{\log x}{\log p} \log p
=\log x \sum_{p \leq x}.
\]
Hence
\[
\frac{\vartheta(x)}{x} \leq \frac{\pi(x) \log x}{x},
\]
whence
\[
\liminf_{x \to \infty} \frac{\vartheta(x)}{x}
\leq
\liminf_{x \to \infty} \frac{\pi(x)}{x/\log x}
\]
and
\[
\limsup_{x \to \infty} \frac{\vartheta(x)}{x}
\leq
\limsup_{x \to \infty} \frac{\pi(x)}{x/\log x}.
\]

Let $0<\alpha<1$. For $x>1$,
\[
\vartheta(x) = \sum_{p \leq x} \log p
\geq \sum_{x^\alpha<p \leq x} \log p
>\sum_{x^\alpha<p \leq x} \log x^\alpha
=\alpha \log x(\pi(x)-\pi(x^\alpha)).
\]
As $\pi(x^\alpha)<x^\alpha$,
\[
\vartheta(x) > \alpha \pi(x) \log x - \alpha x^\alpha \log x,
\]
i.e.,
\[
\frac{\vartheta(x)}{x} > \alpha \frac{\pi(x) \log x}{x} - \alpha \frac{\log x}{x^{1-\alpha}}.
\]
This yields 
\[
\liminf_{x \to \infty} \frac{\vartheta(x)}{x} \geq \alpha \liminf_{x \to \infty}
\frac{\pi(x) \log x}{x} - \alpha
\liminf_{x \to \infty} \frac{\log x}{x^{1-\alpha}}
=\alpha \liminf_{x \to \infty}
\frac{\pi(x) \log x}{x}
\]
and
\[
\limsup_{x \to \infty} \frac{\vartheta(x)}{x} 
\geq 
\alpha \limsup_{x \to \infty}
\frac{\pi(x) \log x}{x} - \alpha
\limsup_{x \to \infty} \frac{\log x}{x^{1-\alpha}}
=\alpha \limsup_{x \to \infty}
\frac{\pi(x) \log x}{x}.
\]
Since these are true for  all $0<\alpha<1$, we obtain respectively
\[
\liminf_{x \to \infty} \frac{\vartheta(x)}{x} \geq 
\liminf_{x \to \infty}
\frac{\pi(x) \log x}{x}
\]
and
\[
\limsup_{x \to \infty} \frac{\vartheta(x)}{x} \geq 
\limsup_{x \to \infty}
\frac{\pi(x) \log x}{x}.
\]
\end{proof}

\section{Wiener's tauberian theorem}
Wiener \cite[Chapter~III]{wiener}.

Wiener-Ikehara \cite{chandrasekharan148}

Rudin \cite[p.~229, Theorem 9.7]{rudin}

We say that a function $s:(0,\infty) \to \mathbb{R}$ is \textbf{slowly decreasing} if 
\[
\liminf (s(\rho v)-s(v)) \geq 0, \qquad v \to \infty, \quad \rho \to 1^+.
\]

Widder \cite[p.~211, Theorem 10b]{widder}: Wiener's tauberian theorem tells us that if
$a \in L^\infty(0,\infty)$ and is slowly decreasing and if $g \in L^1(0,\infty)$ satisfies
\[
\int_0^\infty t^{ix} g(t) dt \neq 0, \qquad x \in \mathbb{R},
\]
then
\[
\lim_{x \to \infty} \frac{1}{x} \int_0^\infty g\left(\frac{t}{x}\right) a(t) dt = A \int_0^\infty g(t) dt
\]
implies that
\[
\lim_{v \to \infty} a(v) = A.
\]

It is straightforward to check the following by rearranging summation.

\begin{lemma}
If $\sum_{n=1}^\infty a_n z^n$ has radius of convergence $\geq 1$, then for $|z|<1$,
\[
\sum_{n=1}^\infty a_n \frac{z^n}{1-z^n} = \sum_{n=1}^\infty \left( \sum_{m | n} a_m\right) z^n.
\]
\label{rearrangement}
\end{lemma}

Using Lemma \ref{rearrangement} with $a_n=\Lambda(n)$ and $z=e^{-x}$ and applying \eqref{mangoldtsum}, we get
\begin{equation}
\sum_{n=1}^\infty \Lambda(n) \frac{z^n}{1-z^n} = \sum_{n=1}^\infty \log(n) z^n.
\label{mangoldtseries}
\end{equation}
From \eqref{mangoldtseries}, and Lemma \ref{rearrangement} with $a_n=1$, we have
\[
\sum_{n=1}^\infty (\Lambda(n)-1) \frac{e^{-nx}}{1-e^{-nx}}
=\sum_{n=1}^\infty (\log n  - d(n)) e^{-nx}.
\]

We follow
Widder \cite[p.~231, Theorem 16.6]{widder}.

\begin{theorem}
As $x \to 0^+$,
\[
\sum_{n=1}^\infty (\log n  - d(n)) e^{-nx}=
 -\frac{2\gamma}{x} + O(x^{-1/2}).
\]
\label{logd}
\end{theorem}
\begin{proof}
Generally,
\begin{align*}
(1-z)\sum_{n=1}^\infty z^n \sum_{m=1}^n a_m&=
(1-z)\sum_{m=1}^\infty a_m \sum_{n=m}^\infty z^n\\
&=(1-z) \sum_{m=1}^\infty a_m \frac{z^m}{1-z}\\
&=\sum_{m=1}^\infty a_m z^m.
\end{align*}
Using this with $a_m = \log m-d(m)$ and $z=e^{-x}$ gives
\begin{align*}
\sum_{n=1}^\infty (\log n-d(n)) e^{-nx}&=
(1-e^{-x})\sum_{n=1}^\infty e^{-nx}\left( \sum_{m=1}^n \log m-\sum_{m=1}^n d(m) \right)\\
&=(1-e^{-x})\sum_{n=1}^\infty e^{-nx} \left( \log (n!) - \sum_{m=1}^n d(m) \right).
\end{align*}
Using
\[
\log(n!)  = n\log n -n + O(\log n)
\]
and
\[
\sum_{m=1}^n d(m) = n\log n + (2\gamma-1)n +O(n^{1/2}),
\]
we get
\[
\log(n!) - \sum_{m=1}^n d(m) = -2\gamma n + O(n^{1/2}).
\]
Therefore,
\[
\sum_{n=1}^\infty (\log n-d(n)) e^{-nx} =(1-e^{-x}) \sum_{n=1}^\infty e^{-nx}( -2\gamma n + O(n^{1/2})).
\]
One proves that there is some $K$ such that for all $0 \leq y < 1$,
\[
(1-y) \left( \log \frac{1}{y} \right)^{1/2} \sum_{n=1}^\infty n^{1/2} y^n \leq K,
\]
whence, with $y=e^{-x}$,
\[
\sum_{n=1}^\infty n^{1/2} e^{-nx} \leq K \frac{x^{-1/2}}{1-e^{-x}}.
\]
Also,
\[
\sum_{n=1}^\infty ne^{-nx} = \frac{e^{-x}}{(1-e^{-x})^2},
\]
and thus we have
\begin{align*}
\sum_{n=1}^\infty (\log n-d(n)) e^{-nx} &=-2\gamma \frac{e^{-x}}{1-e^{-x}}
+O(x^{-1/2})\\
&=-2\gamma \frac{1}{e^x-1} +O(x^{-1/2}).
\end{align*}
But
\[
\frac{1}{e^x-1}  = \frac{1}{x}-\frac{1}{2}+O(x),
\]
so 
\[
\sum_{n=1}^\infty (\log n-d(n)) e^{-nx} = -\frac{2\gamma}{x} + O(x^{-1/2}).
\]
\end{proof}

Define
\[
f(x) = \sum_{n=1}^\infty (\Lambda(n)-1) \frac{e^{-nx}}{1-e^{-nx}},
\]
and
\[
h(x) = \sum_{n \leq x} \frac{\Lambda(n)-1}{n},
\]
and
\[
g(t) = \frac{d}{dt} \left( \frac{te^{-t}}{1-e^{-t}} \right).
\]

First we show that $h$ is slowly decreasing.

\begin{lemma}
$h(x)$ is slowly decreasing.
\end{lemma}
\begin{proof}
Using
\[
\sum_{1 \leq n \leq x} \frac{1}{n} = \log x + \gamma + O(n^{-1}), \qquad x \to \infty,
\]
we have, for $0<x<\infty$ and $\rho>1$,
\begin{align*}
h(\rho x)-h(x)&=\sum_{x<n \leq \rho x} \frac{\Lambda(n)-1}{n}\\
&\geq -\sum_{x<n \leq \rho x} \frac{1}{n}\\
&=-\sum_{1 \leq n \leq \rho x} \frac{1}{n} + \sum_{1 \leq n \leq x} \frac{1}{n}\\
&=-\log(\rho x) +  \log x + O((\rho x)^{-1})  + O(x^{-1})\\
&=-\log \rho + O((\rho x)^{-1})  + O(x^{-1}).
\end{align*}
Hence as $x \to \infty$ and $\rho \to 1^+$, 
\[
h(\rho x)-h(x) \to 0,
\]
which shows that $h$ is slowly decreasing.
\end{proof}

The following is from Widder \cite[pp.~231--232]{widder}.

\begin{lemma}
As $x \to \infty$,
\[
\frac{1}{x} \int_0^\infty g\left(\frac{t}{x} \right) h(t) dt = 2\gamma + O(x^{-1/2}).
\]
\end{lemma}
\begin{proof}
Let $I(t)=0$ for $t < 0$ and $I(t)=1$ for $t \geq 0$. Writing
\[
h(x) = \sum_{n=1}^\infty I(x-n)  \frac{\Lambda(n)-1}{n},
\]
we check that for $x>0$,
\begin{align*}
\int_0^\infty \frac{te^{-xt}}{1-e^{-xt}} dh(t)&=\sum_{n=1}^\infty \int_0^\infty 
\frac{te^{-xt}}{1-e^{-xt}}  \frac{\Lambda(n)-1}{n} d(I(t-n))\\
&=\sum_{n=1}^\infty \int_0^\infty 
\frac{te^{-xt}}{1-e^{-xt}}  \frac{\Lambda(n)-1}{n} d\delta_n(t)\\
&=\sum_{n=1}^\infty \frac{ne^{-nx}}{1-e^{-nx}} \frac{\Lambda(n)-1}{n}\\
&=f(x).
\end{align*}

On the other hand, integrating by parts,
\begin{align*}
f(x)&=\int_0^\infty \frac{te^{-xt}}{1-e^{-xt}} dh(t)\\
&=\int_0^\infty \frac{1}{x} \frac{xt e^{-xt}}{1-e^{xt}} dh(t)\\
&=\int_0^\infty \frac{1}{x} \frac{xt e^{-xt}}{1-e^{-xt}} dh(t)\\
&=\int_0^\infty \frac{1}{x} \frac{t e^{-t}}{1-e^{-t}} dh\left(\frac{t}{x}\right)\\
&=\frac{1}{x} \frac{t e^{-t}}{1-e^{-t}} h\left(\frac{t}{x}\right) \bigg|_0^\infty 
- \int_0^\infty \frac{1}{x} g(t) h\left(\frac{t}{x}\right) dt\\
&=-\int_0^\infty \frac{1}{x} g(t) h\left(\frac{t}{x}\right) dt\\
&=-\int_0^\infty g(xt) h(t) dt.
\end{align*}

By Theorem \ref{logd}, as $x \to 0^+$,
\[
f(x) = -\frac{2\gamma}{x} + O(x^{-1/2}),
\]
i.e., as $x \to 0^+$,
\[
\int_0^\infty g(xt) h(t) dt = \frac{2\gamma}{x} + O(x^{-1/2}).
\]
Thus, as $x \to \infty$,
\[
\int_0^\infty g\left(\frac{t}{x} \right) h(t) dt = 2\gamma x + O(x^{1/2}).
\]
\end{proof}

The following is from Widder \cite[p.~232]{widder}.

\begin{lemma}
\[
\int_0^\infty t^{-ix} g(t) dt =\begin{cases}
-1&x=0\\
 ix \zeta(1-ix)\Gamma(1-ix)&x \neq 0.
 \end{cases}
\]
\end{lemma}
\begin{proof}
\begin{align*}
\int_0^\infty t^{-ix} g(t) dt&=\int_0^\infty t^{-ix}  \frac{d}{dt} \left( \frac{te^{-t}}{1-e^{-t}} \right) dt\\
&=\lim_{\delta \to 0} \int_0^\infty t^{-ix+\delta}  \frac{d}{dt} \left( \frac{te^{-t}}{1-e^{-t}} \right) dt\\
&=\lim_{\delta \to 0} \left( t^{-ix+\delta} \frac{te^{-t}}{1-e^{-t}} \bigg |_0^\infty 
+(ix-\delta) \int_0^\infty  t^{-ix+\delta-1} \frac{te^{-t}}{1-e^{-t}} dt \right)\\
&=\lim_{\delta \to 0} (ix-\delta)  \int_0^\infty  t^{-ix+\delta-1} \frac{te^{-t}}{1-e^{-t}} dt\\
&=\lim_{\delta \to 0} (ix-\delta)  \int_0^\infty   \frac{t^{(-ix+\delta+1)-1}e^{-t}}{1-e^{-t}} dt.
\end{align*}
Using
\[
\int_0^\infty \frac{t^{s-1}}{e^t-1} dt = \zeta(s) \Gamma(s), \qquad \Re(s)>1,
\]
this becomes
\[
\int_0^\infty t^{-ix} g(t) dt=\lim_{\delta \to 0^+} (ix-\delta) \zeta(1+\delta-ix)\Gamma(1+\delta-ix).
\]
If $x=0$, then
using 
\[
\zeta(s) = \frac{1}{s-1}+\gamma +O(|s-1|), \qquad s \to 1,
\]
we get
\[
\lim_{\delta \to 0^+} (-\delta) \zeta(1+\delta)\Gamma(1+\delta) = -1.
\]
If $x>0$, then
\[
\lim_{\delta \to 0^+} (ix-\delta) \zeta(1+\delta-ix)\Gamma(1+\delta-ix) = 
ix\zeta(1-ix)\Gamma(1-ix).
\]
\end{proof}

By Wiener's tauberian theorem, it follows that
\[
\sum_{n=1}^\infty \frac{\Lambda(n)-1}{n} = -2\gamma.
\]

\begin{lemma}
\[
h(x) = \int_{\frac{1}{2}}^x \frac{d(\psi(t)-[t])}{t}.
\]
\end{lemma}
\begin{proof}
Let $I(t)=0$ for $t < 0$ and $I(t)=1$ for $t \geq 0$. Writing
\[
\psi(x) = \sum_{n=1}^\infty I(x-n) \Lambda(n),
\qquad [x]=
\sum_{n=1}^\infty I(x-n),
\]
we have
\begin{align*}
 \int_{\frac{1}{2}}^x \frac{d(\psi(t)-[t])}{t}
&= \int_{\frac{1}{2}}^x 
\frac{1}{t} d\left(\sum_{n=1}^\infty I(t-n) (\Lambda(n)-1)\right)\\
&= \int_{\frac{1}{2}}^x  \frac{1}{t} 
\sum_{n=1}^\infty (\Lambda(n)-1) d \delta_n(t)\\
&=\sum_{1 \leq n \leq x} \frac{\Lambda(n)-1}{n}\\
&=h(x).
\end{align*}
\end{proof}

Thus, we have established that
\[
\int_{\frac{1}{2}}^\infty \frac{d(\psi(t)-[t])}{t} = -2\gamma.
\]

\section{Hermite}
Hermite \cite{hermite1884}

Hermite \cite{hermite}

\section{Gerhardt}
Gerhardt \cite[p.~196]{gerhardt} refers to Lambert's {\em Architectonic}. 

\section{Levi-Civita}
Levi-Civita \cite{levicivita}

\section{Franel}
Franel \cite{franel51} and \cite{franel52}

The next theorem shows that  
 the set of points on the unit circle that are singularities of $\sum_{n=1}^\infty \frac{z^n}{1-z^n}$ is dense in the unit
circle. Titchmarsh \cite[pp.~160--161, \S 4.71]{titchmarsh}.

\begin{theorem}
For $|z|<1$, define
\[
f(z) = \sum_{n=1}^\infty \frac{z^n}{1-z^n}.
\]
Suppose that $p>0, q>1$ are relatively prime integers.
As $r \to 1^-$,
\[
(1-r) f(r e^{2\pi i/q}) \to \infty.
\]  
\end{theorem}
\begin{proof}
Set $z=re^{2\pi i p/q}$ and write
\[
\sum_{n=1}^\infty \frac{z^n}{1-z^n} =
\sum_{n \equiv 0 \pmod{q}}   \frac{z^n}{1-z^n}
+ \sum_{n \not \equiv 0 \pmod{q}}  \frac{z^n}{1-z^n}.
\]

On the one hand, 
\begin{align*}
(1-r) \sum_{n \equiv 0 \pmod{q}}   \frac{z^n}{1-z^n}&=
(1-r) \sum_{m=1}^\infty \frac{z^{mq}}{1-z^{mq}}\\
&=(1-r) \sum_{m=1}^\infty \frac{(re^{2\pi i p/q})^{mq}}{1-(re^{2\pi i p/q})^{mq}}\\
&=(1-r) \sum_{m=1}^\infty \frac{r^{mq}}{1-r^{mq}}\\
&=\frac{1-r}{1-r^q} \sum_{m=1}^\infty \frac{r^{mq}}{1+r^q+\cdots+r^{(m-1)q}}\\
&=\frac{1}{1+r+\cdots+r^{q-1}} \sum_{m=1}^\infty \frac{r^{mq}}{1+r^q+\cdots+r^{(m-1)q}}\\
&\geq \frac{1}{q} \sum_{m=1}^\infty \frac{r^{mq}}{m}\\
&=-\frac{1}{q} \log(1-r^q)\\
&\to \infty
\end{align*}
as $r \to 1$.

On the other hand, for $n \not \equiv 0 \pmod{q}$ we have
\begin{align*}
|1-z^n|^2 &= |1-r^n e^{2\pi ipn/q}|^2\\
& = (1-r^n e^{2\pi ipn/q})(1-r^n e^{-2\pi i pn/q})\\
&=1-r^n(e^{2\pi ipn/q}+e^{-2\pi ipn/q}) +r^{2n}\\
&=1-2r^n \cos 2\pi pn/q + r^{2n}\\
&=1-2r^n + 4r^n \sin^2 \frac{\pi pn}{q} + r^{2n}\\
&= (1-r^n)^2 + 4r^n \sin^2 \frac{\pi pn}{q}.
\end{align*}
So far we have not used the hypothesis that $n \equiv 0 \pmod{q}$. We use it to obtain
\[
\sin \frac{\pi pn}{q} \geq \sin \frac{\pi}{q}.
\]
With this we have
\[
|1-z^n|^2 \geq 4r^n \sin^2 \frac{\pi}{q},
\]
and therefore, as $r<1$,
\begin{align*}
(1-r) \left| \sum_{n \not \equiv 0 \pmod{q}} \frac{z^n}{1-z^n} \right|&\leq
(1-r) \sum_{n \not \equiv 0 \pmod{q}} \frac{|z|^n}{|1-z^n|}\\
&\leq (1-r) \sum_{n \not \equiv 0 \pmod{q}} \frac{r^n}{2r^{n/2} \sin \frac{\pi}{q}}\\
&\leq \frac{1-r}{2 \sin \frac{\pi}{q}}  \sum_{n=0}^\infty r^{n/2}\\
&=\frac{1-r}{2 \sin \frac{\pi}{q}} \cdot \frac{1}{1-\sqrt{r}}\\
&=\frac{1+\sqrt{r}}{2\sin \frac{\pi}{q}}\\
&<\frac{1}{\sin \frac{\pi}{q}}.
\end{align*}
\end{proof}

\section{Wigert}
The following result is proved by Wigert \cite{wigert}. Our proof follows Titchmarsh \cite[p.~163, Theorem 7.15]{zeta}. 
Cf. Landau  \cite{landau1918}.

\begin{theorem}
For $\lambda<\frac{1}{2}\pi$ and $N \geq 1$,
\[
\sum_{n=1}^\infty d(n) e^{-nz} = \frac{\gamma}{z}-\frac{\log z}{z}+\frac{1}{4}- \sum_{n=0}^{N-1} \frac{B_{2n+2}^2}{(2n+2)!(2n+2)} z^{2n+1}
+O(|z|^{2N})
\]
as $z \to 0$ in any angle $|\arg z| \leq \lambda$.
\label{wigert}
\end{theorem}
\begin{proof}
For $\sigma>1$, $s=\sigma+it$,
\[
\zeta^2(s) = \sum_{n=1}^\infty \frac{d(n)}{n^s}.
\]
Using this, for
 $\Re z>0$ we have
\begin{align}
\frac{1}{2\pi i}\int_{2-i\infty}^{2+i\infty}
\Gamma(s) \zeta^2(s) z^{-s} ds & = \sum_{n=1}^\infty d(n) \frac{1}{2\pi i} \int_{2-i\infty}^{2+i\infty}
\Gamma(s) (nz)^{-s} ds
\nonumber \\
& = \sum_{n=1}^\infty d(n) e^{-nz}.
\label{mellin}
\end{align}

Define $F(s) = \Gamma(s) \zeta^s(s) z^{-s}$. $F$ has poles at
$1,0$, and the negative odd integers. (At each negative even integer, $\Gamma$ has a first order pole but $\zeta^2$ has a second order
zero.) First we determine the residue of $F$ at $1$.
We use the asymptotic formula
\[
\zeta(s) = \frac{1}{s-1}+\gamma+O(|s-1|), \qquad s \to 1,
\]
the asymptotic formula
\[
\Gamma(s)=1-\gamma(s-1)+O(|s-1|^2), \qquad s \to 1,
\]
and  the asymptotic formula
\[
z^{-s} = \frac{1}{z}-\frac{\log z}{z} (s-1)+O(|s-1|^2), \qquad  s \to 1,
\]
to obtain
\begin{align*}
 \Gamma(s) \zeta^s(s) z^{-s} &=(1-\gamma(s-1)+O(|s-1|^2)  )
\cdot \left(\frac{1}{(s-1)^2}+\frac{2\gamma}{s-1}+O(|s-1|^2)\right)\\
&\cdot \left( \frac{1}{z}-\frac{\log z}{z} (s-1)+O(|s-1|^2) \right)\\
&=\frac{1}{z(s-1)^2}
-\frac{\gamma}{z(s-1)}+\frac{2\gamma}{z(s-1)}-\frac{\log z}{z(s-1)}+O(1)\\
&=\frac{1}{z(s-1)^2}+\frac{\gamma}{z(s-1)}-\frac{\log z}{z(s-1)}+O(1).
\end{align*}
Hence the residue of $F$ at $1$ is 
\[
\frac{\gamma}{z}-\frac{\log z}{z}.
\]

Now we determine the residue of $F$ at $0$. The residue of $\Gamma$ at $0$ is $1$, and 
hence the residue of $F$ at $0$ is 
\[
1 \cdot \zeta^2(0) \cdot z^0 = \zeta^2(0) = \left( -\frac{1}{2} \right)^2 = \frac{1}{4}.
\]

Finally, for $n \geq 0$ we determine the residue of $F$ at $-(2n+1)$. The residue of $\Gamma$ at $-(2n+1)$ is
$\frac{(-1)^{2n+1}}{(2n+1)!}$, hence the residue of $F$ at $-(2n+1)$ is
\[
\frac{(-1)^{2n+1}}{(2n+1)!} \cdot \zeta^2(2n+1) \cdot z^{2n+1} = 
-\frac{B_{2n+2}^2}{(2n+2)!(2n+2)} z^{2n+1}
\]  
using 
\[
\zeta(-m) = -\frac{B_{m+1}}{m+1}, \qquad m \geq 1.
\]

Let $M>0$,
and let $C$ be the rectangular path  starting at $2-iM$, then going to $2+iM$, then going to $-2N+iM$,  then going to $-2N-iM$, and then ending 
at $2-iM$. 
By the residue theorem,
\begin{equation}
\int_C F(s) ds = 2\pi i\left(\frac{\gamma}{z}-\frac{\log z}{z}+\frac{1}{4} + \sum_{n=0}^{N-1}  -\frac{B_{2n+2}^2}{(2n+2)!(2n+2)} z^{2n+1}
\right).
\label{residue}
\end{equation}
Denote the right-hand sideof \eqref{residue}  by $2\pi i R$.
We have
\[
\int_C F(s) ds = 
\int_{2-iM}^{2+iM} F(s) ds + \int_{2+iM}^{-2N+iM} F(s) ds
+ \int_{-2N+iM}^{-2N-iM} F(s) ds +
\int_{-2N-iM}^{2-iM} F(s) ds.
\]
We shall show that the second and fourth integrals tend to $0$ as $M \to \infty$.
For $s=\sigma+it$ with $-2N \leq \sigma \leq 2$, Stirling's formula \cite[p.~151]{titchmarsh}  tells us that
\[
|\Gamma(s)| \sim \sqrt{2\pi} e^{-\frac{\pi}{2} |t|} |t|^{\sigma-\frac{1}{2}}, \qquad |t| \to \infty.
\]
As well \cite[p.~95]{zeta},
there is some $K>0$  such that in the half-plane $\sigma \geq -2N$,
\[
\zeta(s)=O(|t|^K).
\]
Also,
\begin{align*}
z^{-s} &= e^{-s \log z}\\
& = e^{-(\sigma+it)(\log |z|+i\arg z)} \\
&= e^{-\sigma \log|z| + t \arg z - i(\sigma \arg z+t \log|z|)},
\end{align*}
and so for $|\arg z| \leq \lambda$,
\[
|z^{-s}| = e^{-\sigma \log|z| + t \arg z}  \leq e^{-\sigma \log |z|+\lambda |t|} = |z|^{-\sigma} e^{\lambda |t|}.
\]
Therefore
\[
\left| \int_{2+iM}^{-2N+iM} F(s) ds \right| 
\leq (2+2N) \sup_{-2N \leq \sigma \leq 2} |F(\sigma+iM)|
=O(e^{-\frac{\pi}{2} M} M^{\sigma-\frac{1}{2}} M^{2K}  |z|^{-\sigma} e^{\lambda M}),
\]
and because $\lambda<\frac{\pi}{2}$ this tends to $0$ as $M \to \infty$.
Likewise,
\[
\left|\int_{-2N-iM}^{2-iM} F(s) ds \right| \to 0
\]
as $M \to \infty$. It follows that
\[
\int_{2-i\infty}^{2+i\infty} F(s) ds +\int_{-2N+i\infty}^{-2N-i\infty} F(s) ds
=2\pi i R.
\]
Hence,
\[
\int_{2-i\infty}^{2+i\infty} F(s) ds = 2\pi i R + \int_{-2N-i\infty}^{-2N+i\infty} F(s) ds.
\]
We bound the integral on the right-hand side. We have
\[
 \int_{-2N-i\infty}^{-2N+i\infty} F(s) ds = 
  \int_{\sigma=-2N, |t| \leq 1} F(s) ds+
    \int_{\sigma=-2N, |t| > 1} F(s)ds.
\]
The first integral satisfies
\[
\left|   \int_{\sigma=-2N, |t| \leq 1} F(s) ds \right|
\leq  \int_{\sigma=-2N, |t| \leq 1} |\Gamma(s) \zeta^2(s)|  |z|^{-\sigma} e^{\lambda |t|}
ds
= |z|^{2N} \cdot O(1) = O(|z|^{2N}),
\]
because $\Gamma(s) \zeta^2(s)$ is continuous on the path of integration.
The second integral satisfies
\begin{align*}
\left| \int_{\sigma=-2N, |t|>1} F(s) ds \right| &
\leq
 \int_{\sigma=-2N, |t|>1}
e^{-\frac{\pi}{2} |t|} |t|^{\sigma-\frac{1}{2}} |t|^K |z|^{-\sigma} e^{\lambda |t|}  ds\\
&=|z|^{2N}  \int_{\sigma=-2N, |t|>1} e^{-\frac{\pi}{2} |t|} |t|^{-2N-\frac{1}{2}} |t|^K  e^{\lambda |t|}  dt\\
&=|z|^{2N} \cdot O(1)\\
&=O(|z|^{2N}),
\end{align*}
because $\lambda<\frac{\pi}{2}$. This establishes
\[
\frac{1}{2\pi i} \int_{2-i\infty}^{2+i\infty} F(s) ds = R +O(|z|^{2N}).
\]
Using \eqref{mellin} and \eqref{residue}, this becomes
\[
\sum_{n=1}^\infty d(n) e^{-nz} = 
\frac{\gamma}{z}-\frac{\log z}{z}+\frac{1}{4} - \sum_{n=0}^{N-1}  \frac{B_{2n+2}^2}{(2n+2)!(2n+2)} z^{2n+1}
+O(|z|^{-2N}),
\]
completing the proof.
\end{proof}

For example,
as $B_2=\frac{1}{6},B_4=-\frac{1}{30},B_6=\frac{1}{42}$, the above theorem tells us that
\[
\sum_{n=1}^\infty d(n)e^{-nz} =  \frac{\gamma}{z}-\frac{\log z}{z}+\frac{1}{4}
-\frac{z}{144}
-\frac{z^3}{86400}
-\frac{z^5}{7620480}+
O(|z|^6).
\]

\section{Steffensen}
Steffensen \cite{steffensen}

\section{Szeg{\H o}}
Szeg{\H o} \cite{szego}

\section{P\'olya and Szeg{\H o}}
P\'olya and Szeg{\H o} \cite{polya}

\section{Partition function}
Let 
\[
F(x)=\sum_{n=0}^\infty p(n) x^n = \prod_{n=1}^\infty \frac{1}{1-x^n}.
\]
Taking the logarithm,
\[
\log F(x)=\sum_{n=1}^\infty \log \frac{1}{1-x^n}
=-\sum_{n=1}^\infty \log(1-x^n)
=-\sum_{n=1}^\infty \sum_{m=1}^\infty -\frac{(x^n)^m}{m},
\]
and switching the order of summation gives
\[
\log F(x)
=\sum_{m=1}^\infty \frac{1}{m} \sum_{n=1}^\infty (x^m)^n
=\sum_{m=1}^\infty \frac{1}{m} \frac{x^m}{1-x^m}.
\]
On the one hand, for $0<x<1$ we have $mx^{m-1}(1-x)<1-x^m$ and using this,
\[
\sum_{m=1}^\infty \frac{1}{m} \frac{x^m}{1-x^m}<\sum_{m=1}^\infty \frac{1}{m} \frac{x^m}{mx^{m-1}(1-x)}
=\frac{x}{1-x} \sum_{m=1}^\infty \frac{1}{m^2}=\frac{\pi^2}{6} \frac{x}{1-x}.
\]
On the other hand, for $-1<x<1$ we have $1-x^m<m(1-x)$, and using this, for $0<x<1$ we have
\[
\sum_{m=1}^\infty \frac{1}{m} \frac{x^m}{1-x^m}
>\sum_{m=1}^\infty \frac{1}{m} \frac{x^m}{m(1-x)}
=\frac{1}{1-x} \sum_{m=1}^\infty \frac{x^m}{m^2}.
\]
Thus, for $0<x<1$,
\[
\sum_{m=1}^\infty \frac{x^m}{m^2} < (1-x) \log F(x) < \frac{\pi^2}{6}x.
\]
Taking $x \to 1^{-}$ gives
\[
\frac{\pi^2}{6} \leq \lim_{x \to 1^-} (1-x)\log F(x) \leq \frac{\pi^2}{6},
\]
i.e.,
\[
\log F(x) \sim \frac{\pi^2}{6} \frac{1}{1-x}, \qquad x \to 1^-.
\]
See Stein and Shakarchi \cite[p.~311]{steinII}.

\section{Hansen}
Hansen \cite{hansen}

\section{Kiseljak}
Kiseljak \cite{kiseljak}

\section{Unsorted}
In 1892, in volume~VII, no.~23, 
p.~296 of the weekly {\em Naturwissenschaftliche Rundschau}, it is stated that for the year 1893, one of the six prize questions
for the 
Belgian Academy of Sciences in Brussels
 is
to determine the sum of the Lambert series
\[
\frac{x}{1-x}+\frac{x^2}{1-x^2}+\frac{x^3}{1-x^3}+\cdots,
\]
or if one cannot do this, to find a differential equation that determines the function.

Gram \cite{gram} on distribution of prime numbers.

Hardy \cite{divergent}

Bohr and Cramer \cite[p.~820]{bohr}

Flajolet, Gourdon and Dumas \cite{flajolet}

\bibliographystyle{plain}
\bibliography{lambert}

\begin{thebibliography}{10}
\bibitem{glaisher}
J.~W.~L. Glaisher.
\newblock On the square of the series in which the coefficients are the sums of
  the divisors of the exponents.
\newblock {\em Messenger of Mathematics}, 14:156--163, 1885.
\end{thebibliography}

\end{document}